\documentclass{sig-alternate}

\makeatletter
\newif\if@restonecol
\makeatother

\usepackage{amsfonts}
\usepackage{amsmath}
\usepackage{amssymb}
\usepackage[algoruled,vlined,linesnumbered]{algorithm2e}

\DeclareFixedFont{\auacc}{OT1}{phv}{m}{n}{12}   
\DeclareFixedFont{\afacc}{OT1}{phv}{m}{n}{10}   
\begin{document}

\title{Real root refinements for univariate polynomial equations}

%
%
%
%
%

\numberofauthors{1} 
%

\author{
%
%
\alignauthor
Ye Liang \\
       \affaddr{KLMM, Institute of Systems Science,}\\
       \affaddr{Academy of Mathematics and System Science,}\\
       \affaddr{Chinese Academy of Sciences,}\\
       \affaddr{55 Zhongguancun East Road, Haidian District, 100190 Beijing,
       CHINA}\\
       \email{wolf39150422@gmail.com}
}

\maketitle

\begin{abstract}
Real root finding of polynomial equations is a basic problem in
computer algebra. This task is usually divided into two parts:
isolation and refinement. In this paper, we propose two algorithms
LZ1 and LZ2 to refine real roots of univariate polynomial equations.
Our algorithms combine Newton's method and the secant method to
bound the unique solution in an interval of a monotonic convex
isolation (MCI) of a polynomial, and have quadratic and cubic
convergence rates, respectively. To avoid the swell of coefficients
and speed up the computation, we implement the two algorithms by
using the floating-point interval method in \verb"Maple15" with the
package \verb"intpakX". Experiments show that our methods are
effective and much faster than the function \verb"RefineBox" in the
software \verb"Maple15" on benchmark polynomials.
\end{abstract}


\keywords{Real root refinement, Newton's method, floating-point,\\ interval method} 

\newtheorem{theorem}{Theorem}[section]
\newtheorem{proposition}[theorem]{Proposition}
\newtheorem{lemma}[theorem]{Lemma}
\newtheorem{definition}[theorem]{Definition}
\newtheorem{observation}[theorem]{Observation}
\newtheorem{corollary}[theorem]{Corollary}
\newtheorem{problem}[theorem]{Problem}
\newtheorem{remark}[theorem]{Remark}
\newtheorem{example}[theorem]{Example}
\newtheorem{assumption}[theorem]{Assumption}

\section{Introduction}
Solving for real solutions of univariate polynomial equations is a
fundamental problem in computer algebra. Many other problems in
mathematics or other fields can be reduced to it such as real
solving multivariate polynomial equations
\cite{Rouillier99,ChengGaoGuo11,ChengGaoYap09, BCLM09}, studying the
topologies of real algebraic plane curves \cite{CLPPRT10}, and
generating ray-traced images of implicit surfaces in computer
graphics \cite{Mitchell90}. There is a vast literature on
calculating real zeros of univariate polynomials. We refer to
\cite{Pan97} for some of the references.

The process of reliable computing real roots is usually divided into
two steps: \emph{real root isolation} (i.e., cutting the real axis
so that each real root is contained in a separate interval) and
\emph{real root refinement} (i.e., narrowing each isolating interval
to a given width). We mainly concerned with how to efficiently
refine each real root of a univariate polynomial equation to a high
precision. We propose two algorithms LZ1 (Algorithm \ref{LZ1}) and
LZ2 (Algorithm \ref{LZ2}) to refine real roots.

LZ1 is a combination of Newton's method and the secant method. It is
based on Theorem \ref{BasicTheorem} (cf. Theorem 4.6 in
\cite{Yan92}) that can help choose a starting point and guarantee
the quadratic convergence of the point sequence for Newton's method
in an interval. By this theorem, all of the points in the sequence
locate on the same side of the real root $\xi$, i.e., each point in
the sequence provides a bound of $\xi$ of the same kind (upper bound
or lower bound). For each point in the sequence, to get a bound of
$\xi$ of the other kind, the secant method should be applied.
Theorem \ref{Theorem:LZ1} shows that LZ1 has an at least quadratic
convergence rate.

LZ2 is also a combination of Newton's method and the secant method.
But it makes an opposite choice of the starting point for Newton's
method to LZ1. As a result, if we consider the starting point as a
bound of $\xi$, then the point found by Newton's method becomes
another bound of $\xi$ of the other kind. After that, by the secant
method, a new bound of $\xi$ of the same kind with the starting
point can be obtained. In this way, the two kinds of bounds of $\xi$
arise alternately. Thus, the precision of each bound of $\xi$ can
benefit from its predecessor in this point sequence. So, it is not
surprising that the sequence of inclusion intervals consisting of
these bounds converges to $\xi$ at least cubically (cf. Theorem
\ref{Theorem:LZ2}).

However, the implementations of LZ1 and LZ2 will become slow if we
evaluate polynomials in LZ1 and LZ2 exactly, since the
representations of the numbers will swell dramatically. Fortunately,
the floating-point interval method can help solve this problem. But
this method can only deal with the so called well-posed problems in
numerical computation. In fact, to apply this method to speed up the
computation, we have already treated all the ill-posed cases by
exact methods before calling LZ1 or LZ2.

To this aim, a square-free decomposition should be done first to get
a list of square-free polynomials as components and to know the
multiplicity of corresponding roots in each component. Then, we make
a local monotonic convex decomposition (LMCD, cf. Definition
\ref{LMCD}) for each of these square-free polynomials to make sure
each component has a monotonic convex isolation (MCI, cf. Definition
\ref{MCI}). After that, based on existing methods for real root
isolation
\cite{CollinsAkritas76,ABS94,XiaYang02,RZ04,AS05,MRR05,TsigaridasEmiris06,BartonJuttler07,Pan07},
we compute a MCI for any polynomial in a LMCD of every square-free
polynomial obtained in the first step. Finally, floating-point
interval versions of LZ1 and LZ2 can be called safely in theory.

We have implemented LZ1 and LZ2 in \verb"Maple15" with the package
\verb"intpakX" which contains functions for floating-point interval
computation in arbitrary precisions. Experiments were done on
Chebyshev polynomials of the first kind to compare the efficiencies
of LZ1, LZ2 and the \verb"Maple" function \verb"RefineBox". The
timings show that our implementations of LZ1 and LZ2 are much faster
than \verb"RefineBox" and that LZ2 is usually faster than LZ1.

Some variants of Newton's method also exist to reliably refine real
roots of univariate polynomial equations.

Ramon E. Moore \cite{Moore66} in 1966 gave an interval Newton's
method to compute a real root of a univariate real function in a
closed interval. This method can be speeded up by using
floating-point interval arithmetics \cite{Rump10}. It can give
verified results, but sometimes it may fail, e.g. the resulting
interval may contain the original one.

George E. Collins and Werner Krandick \cite{CollinsKrandick93} in
1993 presented three versions of a variant of Newton's method and
proved that the exact version has quadratic convergence rate. Among
the three algorithms, they showed that the one using floating-point
interval arithmetics performed much efficiently than the other two.
Note that the restrictions of the input inclusion interval and the
selection of the starting point for Newton's method in the exact
version of their method are also the same with that in Theorem
\ref{BasicTheorem}. But instead of a secant, when it is needed in
LZ1, they used the line parallel to the tangent at the point that
Newton's method is applied. Moreover, their method needs the input
interval ro satisfy a ``3/4-assumption" to get started.

John Abbott \cite{Abbott06} in 2006 proposed an algorithm named QIR.
Then, based on the QIR, Michael Kerber and Michael Sagraloff
\cite{KerberSagraloff11} in 2011 gave an algorithm EQIR and analyzed
its bit complexity. QIR and EQIR are combinations of the secant
method and the bisection method. One feature of them is that they do
not need any information of derivatives.

The rest of the paper is structured as follows. In Section
\ref{Preliminaries}, we introduce the basic theorem that LZ1 is
based on. Section \ref{Sec:Multiplicities} is devoted to describing
how multiplicities and the square-free polynomial components should
be calculated. The notions of MCI and LMCD are given in Section
\ref{Monotonic convex isolation} in order to get the intervals that
can be used as inputs of LZ1 and LZ2. In Section \ref{Refinement},
the main algorithms LZ1 and LZ2 are presented and their superlinear
convergence properties are proved. Then we discuss how to use
floating-point interval method to speed up LZ1 and LZ2 in
\mbox{Section \ref{Speedup}}. At last we make comparisons on
efficiency among LZ1, LZ2 and the \verb"Maple" function
\verb"RefineBox" in Section \ref{Experiments}.

\section{Basics}\label{Preliminaries}
This section is about a theorem which provides the basis of
Algorithm \ref{LZ1} named LZ1 in Section \ref{Refinement}.


Newton's method is a fast algorithm for computing solutions of
equations. For a function and a stating point, generally, it is not
easy to determine whether Newton's method converges or not. However,
for univariate real functions there exists such a theorem as follows
(cf. \mbox{Theorem 4.6} in \cite{Yan92} written by Qingjin Yan in
1992).

\begin{theorem} \label{BasicTheorem}
If a function $f \in C^2[a,b]$ and satisfies the following
conditions:
\begin{enumerate}
  \item $f(a)f(b)<0$;
  \item the sign of $f''(x)$ does not vary on $[a,b]$;
  \item $f'(x) \neq 0$ when $x \in [a,b]$;
  \item $x_0 \in [a,b]$, $f(x_0)f''(x_0)>0$,
\end{enumerate}
then the sequence $\{x_k\}$ generated by $$ x_{k+1}=x_k -
\frac{f(x_k)}{f'(x_k)}, k=0,1, \ldots$$ monotonically converges to
the unique real root of $f(x)=0$ in $[a,b]$ and the convergence rate
is at least $2$.
\end{theorem}

\begin{remark}
Note that the second condition of Theorem \ref{BasicTheorem} allows
$f''(x)$ to reach zero when $x \in [a,b]$. If we restrict this
condition to ``$f''(x) \neq 0$ when $x \in [a,b]$", then the
iteration sequence will strictly monotonically converges.
\end{remark}



\section{Multiplicities} \label{Sec:Multiplicities}

We briefly recall how to compute multiplicities of roots of
polynomial equations by existing methods.

Though numerical tools exist (e.g. \cite{WZ08,LNZ10}) for computing
multiplicities of roots of polynomial equations and most probably
their outputs are correct, it is still possible that they output
incorrect results sometimes. In other words, for a univariate
polynomial $f \in \mathbb{Q}[x]$, computing the multiplicities of
real or complex roots of the equation $f=0$ is an ill-posed problem
in numerical computation \cite{Rump10}.

Thus, to know these multiplicities exactly, we should use exact
methods. We can make a square-free decomposition of $f$ such that
there exist $s$ polynomials $g_1,\ldots,g_s$ in $\mathbb{Q}[x]$ with
$f=g_1^{r_1} \ldots g_s^{r_s}$ where $1 \leq r_1 < \cdots < r_s \leq
n$ and $(g_i,g_j)=1$ when $i \neq j$. Then $f=0$ has a (complex or
real) root $\xi$ with multiplicity $r$ if and only if there exists a
polynomial $g_i$ in the square-free decomposition of $f$ such that
$r=r_i$ and $\xi$ is a single root of $g_i=0$. The square-free
decomposition of a univariate polynomial is easy to compute in the
software \verb"Maple15" with the function \verb"sqrfree". A
factorization with the function \verb"factor" in the same software
can also help find the multiplicities of corresponding polynomial
zeros. In practice, a square-free decomposition or a factorization
costs much less time than real root isolations and refinements and
can help reduce an reducible polynomial in $\mathbb{Q}[x]$ with high
degree to some lower degree polynomials which are relative easier to
deal with.

In the rest of this paper, we assume that $f$ is a square-free
polynomial with rational coefficients.

\section{Monotonic convex isolations} \label{Monotonic convex isolation}
To refine a real root of a square-free polynomial $f \in
\mathbb{Q}[x]$, it should be first isolated from other real roots.
In this section, we give a concept called monotonic convex isolation
(Definition \ref{Def:Monotonic}). If $f$ has such an isolation, then
any real root of $f$ is contained in a closed interval where the
monotonicity and convexity of $f$ is fixed unless the interval is a
point. The refinements of real roots in the next section will
benefit a lot from the properties of monotonicity and convexity of
$f$ on these isolating intervals. However, certain square-free
polynomials in $\mathbb{Q}[x]$ have no monotonic convex isolations.
In this case, we want to decompose $f$ to a multiplication of some
polynomials that have monotonic convex isolations (Theorem
\ref{Th:decomposition}). This leads to the concept of local
monotonic convex decomposition (Definition \ref{Def:Local}). At the
end of this section, we give an algorithm to compute such a
decomposition.

\begin{definition}[MCI] \label{Def:Monotonic}
Given a real root isolation $Iso$ of a square-free polynomial $f \in
\mathbb{Q}[x]$, we say that $Iso$ is a monotonic convex isolation of
$f$ if $I \in Iso$ is not a point implies that $f'(x) \neq 0$ and
$f''(x) \neq 0$ for all $x \in I$.
\end{definition}

\begin{remark}
``$I \in Iso$ is not a point" in the above definition means that the
closed interval $I$ is not in the form of $[a,a]$ where $a$ is a
rational number. Note that such stuff may exist in $Iso$.
\end{remark}

If $f$ has a MCI, then we can work it out by \mbox{Algorithm
\ref{MCI}}. Note that Algorithm \ref{MCI} is not a definite
algorithm but a description of a class of algorithms. When the base
algorithm for real root isolation (e.g. the \verb"RS" C-library
\cite{RS,Rouillier99,RZ04} of Fabrice Rouillier and
\verb"DISCOVERER" \cite{DISCOVERER,Xia07} of Bican Xia) is chosen,
then \mbox{Algorithm \ref{MCI}} will become definite. Hence, its
efficiency mainly depends on the chosen base algorithm.

\begin{algorithm}[h] \label{MCI}
  \SetLine
  \KwData{A polynomial $f \in \mathbb{Q}[x]$ that has a MCI}
  \KwResult{A monotonic convex isolation of $f$}
  Compute a real root isolation $Iso$ of $f$ in common sense by a base algorithm.

  For each interval $I \in Iso$ that is not a point, cut $I$ by bisection method and
  consider the sign of the value of $f$ at the midpoint of $I$ to find the interval that contains the real root in
  $I$ (sometimes, the midpoint is just the real root).

  Repeat step 2 until neither $f'$ nor $f''$ has real roots in the
  resulting closed interval $I^*$ (this can be done by using inclusion and exclusion criteria).

  Return the set consisting of all the $I^*$ in step 3 and all the point intervals arising in $Iso$ or during the computation of step 3.
  \caption{MCI}
\end{algorithm}

It is obvious that for a nonconstant square-free polynomial $f$, if
it satisfies $(f,f'')=1$ then $f$ has a MCI. However, not all real
polynomials have monotonic convex isolations. We give a
counterexample below.

\begin{example} \label{Exam:ill}
Pick $f$ as $x(x+1)(x+2)$. Then $f'=3x^2+6x+2$ and $f''=6(x+1)$. It
is easy to see that $-1$ is the common real root of $f$ and $f''$.
Obviously, $f$ has no MCI.
\end{example}

Hence, we want a decomposition of $f$ such that each component has a
MCI. This idea delivers the following definition.


\begin{definition}[LMCD] \label{Def:Local}
If a square-free polynomial $f \in \mathbb{Q}[x]$ can be represented
as the multiplication of $s$ polynomials $g_1, \ldots, g_s$ where
$\deg(g_i)=1$ or $\gcd(g_i,g_i'')=1$ for every $i=1,\ldots,s$, then
we say that $f$ has a local monotonic convex decomposition.
\end{definition}

\begin{theorem}\label{Th:decomposition}
Every nonconstant square-free polynomial $f \in \mathbb{Q}[x]$ has a
local monotonic convex decomposition.
\end{theorem}
\begin{proof}
For a polynomial $f \in \mathbb{Q}[x] \setminus \mathbb{Q}$, we
prove the theorem by induction on the degree $d$ of $f$.

When $d=1$, the theorem holds obviously. Supposing the conclusion of
the theorem holds when $d \leq k$ ($k \geq 1$), we
  prove that it also holds for $d=k+1$. Denote $g:=\gcd(f,f'')$. It is easy to see that
  $\deg(g) < \deg(f)$ when $d>1$. If $g=1$ then the
  conclusion holds. Otherwise, there exists a polynomial $h \in \mathbb{Q}[x] \setminus
  \mathbb{Q}$ such that $f=gh$. Since $\deg(g)$ and $\deg(h)$ are all
  no larger than $k$, we know that $g$ and $h$ all have local
  monotonic convex decompositions by the induction assumption. Thus, $f$ has a local
  monotonic convex decomposition.

Therefore, the conclusion of the theorem holds for all nonconstant
square-free polynomials in $\mathbb{Q}[x]$.
\end{proof}

\begin{corollary} \label{Cor:irreducible}
Any nonconstant irreducible polynomial $f \in \mathbb{Q}[x]$ forms a
LMCD of itself.
\end{corollary}

\begin{remark}
Sometimes, the irreducibility of a polynomial $f$ in $\mathbb{Q}[x]$
is easy to test by certain irreducibility criteria, e.g.
Eisenstein's criterion. If the irreducible decomposition of $f$ in
$\mathbb{Q}[x]$ is known, then it is also a monotonic convex
decomposition of $f$ according to Corollary \ref{Cor:irreducible}.
\end{remark}

Based on the proof of Theorem \ref{Th:decomposition}, we give
\mbox{Algorithm \ref{LMCD}} (LMCD) to compute a local monotonic
convex decomposition of $f$. The termination and correctness are
obvious and we omit the proofs. In most cases, the algorithm LMCD
only tests that $f$ and $f''$ are coprime. This test is usually very
fast in practice.

\begin{algorithm}[h] \label{LMCD}
  \SetLine
  \KwData{A nonconstant square-free polynomial $f \in \mathbb{Q}[x]$}
  \KwResult{A local monotonic convex decomposition of $f$}
  $S:=\{\}$\;
  \If{$\deg(f)=1$}{$S:= \{f\}$\; \KwRet $S$\;}
  $g:=\gcd(f,f'')$\;
  \If{$g=1$}{$S:=\{f\}$\; \KwRet $S$\;}
  $S:=S \cup \textup{LMCD}(g) \cup \textup{LMCD}(f/g)$\;
  \KwRet $S$\;
  \caption{LMCD}
\end{algorithm}

\begin{example}
For the polynomial $f=x^3+3x^2+2x$ in Example \ref{Exam:ill}, it can
be decomposed into two polynomials by Algorithm \ref{LMCD}, i.e.,
the output is $\{x+1,x^2+2x\}$. For $x+1=0$, we can directly compute
the root. For $x^2+2x$, it has a monotonic convex isolation.
\end{example}

\section{Real Root Refinements} \label{Refinement}

In this section, we study how Newton's method and the secant method
can be combined and applied on the MCI intervals to compute narrower
inclusion intervals of the real roots of a univariate polynomial
equation. For this, we provide two algorithms LZ1 (Algorithm
\ref{LZ1}) and LZ2 (Algorithm \ref{LZ2}), and prove that their
convergence rates are at leat $2$ and $3$, respectively.

We design LZ1 (Algorithm \ref{LZ1}) based on \mbox{Theorem
\ref{BasicTheorem}}. According to this theorem, the initial point
$x_0$ for Newton's method should satisfy the condition
$f(x_0)f''(x_0)>0$, and the point sequence $\{x_i\}_{i=0}^{\infty}$
of Newton's method converges monotonically. This provides a series
of bounds of the real root $\xi$ on one side. Since the convexity
does not change in the input MCI interval $[a,b]$, the secant method
can be used to obtain a bound $c_{i+1}$ on the other side of $\xi$
after getting each $x_i$ by Newton's method. Then a sequence of
inclusion intervals of $\xi$ is obtained. \mbox{Theorem
\ref{Theorem:LZ1}} shows that this sequence converges at least
quadratically.

The second main algorithm LZ2 (Algorithm \ref{LZ2}) is also a
combination of Newton's method and the secant method. In contrast to
LZ1, the condition for selecting initial point $c_0$ for Newton's
iteration in LZ2 is $f(c_0)f''(c_0)<0$. As a result, the point $z$
($x_1$) obtained by Newton's method on $f$ at $c_0$ locates at the
other side of $\xi$. Then we can get $c_1$ as a bound of $\xi$ at
the same side with $c_0$ by the secant method. In this way, the
inclusion interval sequence of $\xi$ can obtain an at least cubic
convergence rate (cf. \mbox{Theorem \ref{Theorem:LZ2}}). However,
note that when $z \not \in [a,b]$, the behavior of $f$ at $z$ will
be out of control. In this case, $[a,b]$ should be cut until the new
$z$ locates in the new initial MCI interval. This task is easy to do
by using the secant method as described in lines
\ref{loop2begin}-\ref{loop2end} of Algorithm \ref{LZ2}. Then the
fast convergent iteration can start.

We want the outputs of LZ1 and LZ2 to be narrow intervals from which
the precisions of the corresponding real roots can be read out
directly. However, no matter how narrow an interval containing zero
is, we cannot obtain any correct bits by comparing the two endpoints
of the interval. Hence, LZ1 and LZ2 are not allowed to input
intervals that contain zero. In fact, when a monotonic convex
isolation of a square-free polynomial $f$ is performed as we
discussed in the last section, it is very easy to check whether
$f=0$ has zero solution or not. If such solution exists, its
isolation interval can be set as the interval $[0,0]$ in the MCI
directly.

\begin{example} \label{Example:LZ1}
Given $L=8$ (decimal digits), a square-free polynomial
$f:=x^3-20x+7$ and an interval
$[\frac{4389}{1024},\frac{1097}{256}]$ in a monotonic convex
isolation of $f$, we show the process of the iteration of Algorithm
\ref{LZ1} as follows.
\begin{align*}
[a_0,b_0] & = [\frac{4389}{1024},\frac{1097}{256}] \\
[a_1,b_1] & =[\frac{40379863349}{9422150912},\frac{80788619485}{18851042816}]\\
[a_2,b_2] & =
[\frac{18537520582609738516073933520140272183127}{4325505299448020965613362785214886187776},\\
          & \frac{503844952756268930733017297728389}{117566100791919345105185727965440}]\\
\end{align*}

$(b_2-a_2)/b_2 \approx 0.6053885328 \times 10^{-14} < 10^{-8}$.
Hence, the algorithm terminates and outputs $[a_2,b_2]$.

\end{example}

\begin{algorithm}[h] \label{LZ1}
  \SetLine
  \KwData{A closed interval $[a,b]$ ($ab>0$ and $a<b$) in a monotonic convex isolation
  of a square-free polynomial $f \in \mathbb{Q}[x]$ and a positive integer
  $L$}

  \KwResult{A closed interval $[a',b']$ such that the only real root $\xi$ of $f=0$ in $[a,b]$ is also
  in it and $(b'-a')/|\xi| \leq 10^{-L}$}
  \If{$|b-a| \leq 10^{-L} \min(|a|,|b|)$}{\KwRet $[a,b]$;}
  \eIf{$f(a)f''(a)>0$}{$x \longleftarrow a$; $c \longleftarrow b$;}
  {$x \longleftarrow b$; $c \longleftarrow a$;}
  \While{$|x-c|> 10^{-L} \min(|x|,|c|)$}{ \label{loop1begin}
    $u \longleftarrow f(x)$\;
    $v \longleftarrow f(c)$\;
    $p \longleftarrow x-u/f'(x)$\;
    $c \longleftarrow (xv-cu)/(v-u)$ \;
    $x \longleftarrow p$\;
    }\label{loop1end}

    \KwRet $[\min(x,c),\max(x,c)]$.

  \caption{LZ1}
\end{algorithm}

\begin{theorem} \label{Theorem:LZ1}
The iteration in Algorithm \ref{LZ1} converges and the convergence
rate is at least 2.
\end{theorem}
\begin{proof} From the loop in lines \ref{loop1begin}-\ref{loop1end} of Algorithm
\ref{LZ1}, we know that the iteration is
\begin{align}
x_{i+1}&=x_i- \frac{f(x_i)}{f'(x_i)} \label{Iteration:Newton} \\
c_{i+1}&=\frac{x_{i}f(c_i) - c_if(x_{i})}{f(c_i)-f(x_{i})}
\label{Iteration:Secant}
\end{align}
where $x_i$ and $c_i$ means the $i$-th iteration of $x$ and $c$ for
$i=0,1,\ldots$, respectively.

We first prove that the real root $\xi$ of the equation $f=0$ always
lies in the open interval between $x_i$ and $c_i$. By Theorem
\ref{BasicTheorem}, we know that the sequence
$\{x_i\}_{i=0}^{\infty}$ converges monotonically to $\xi$ and that
no $x_i$ is equal to $\xi$. Thus, these $x_i$'s are the bounds of
$\xi$ on the same side. We only need to prove that the $c_i$'s are
the bounds of $\xi$ on the other side and $c_i \neq \xi$. Note that
if $c_i$ and $x_{i}$ are on different sides of $\xi$, then
$f(c_i)/(f(c_i)-f(x_{i}))$ and $-f(x_{i})/(f(c_i)-f(x_{i}))$ are all
in $(0,1)$ and their sum is $1$. By induction on $i$, applying
Jessen's inequality we have that $f(c_{i+1})<0$ when $f''(a)>0$ and
$f(c_{i+1})>0$ when $f''(a)<0$, i.e.,
$$f(c_{i+1})f''(a) <0;$$ consequently, $c_{i+1}$ locates on the opposite side of $x_{i+1}$
w.r.t. $\xi$ and $c_{i+1} \neq \xi$. This proves that $\xi$ always
lies in the open interval between $x_{i}$ and $c_{i}$.

Next, we study the convergence of the sequence consisting of the
interval lengths. According to the Mean Value Theorem, there exist a
$\lambda \in (\min\{x_i, \xi \}, \max\{x_i, \xi \})$ and a $\mu \in
(\min\{c_i,x_i\},\max\{c_i,x_i\})$ such that
$f(x_i)=f'(\lambda)(x_i-\xi)$ and
$\frac{f(c_i)-f(x_i)}{c_i-x_i}=f'(\mu)$. Note that the sequence
$\{c_i\}_{i=0}^{\infty}$ converges to $\xi$ monotonically. Indeed,
since $c_{i+1}$ and $c_i$ are on the same side of $\xi$ and
\begin{align*}
c_{i+1}-c_i = \frac{(x_i-c_i)f(c_i)}{f(c_i)-f(x_i)}
            = -\frac{f'(\tau)}{f'(\mu)}(c_i-\xi),
\end{align*}
we know that $\{c_i\}_{i=0}^{\infty}$ is monotone and has a bound
$\xi$. Then $\{c_i\}_{i=0}^{\infty}$ converges and it is easy to see
the limit is $\xi$. Since there exists $\varphi$ between $c_i$ and
$x_i$ such that
$f(c_i)=f(x_i)+(c_i-x_i)f'(x_i)+(c_i-x_i)^2f''(\varphi)/2$, we have
\begin{align*}
x_{i+1}-c_{i+1}&=x_{i}-\frac{f(x_i)}{f'(x_i)}-\frac{x_if(c_i)-c_if(x_i)}{f(c_i)-f(x_i)}\\
 &=f(x_i)(\frac{1}{\frac{f(c_i)-f(x_i)}{c_i-x_i}}-\frac{1}{f'(x_i)})\\
 &=\frac{f(x_i)f''(\varphi)}{2f'(\mu)f'(x_i)}(x_i-c_i)\\
 &=\frac{f''(\varphi)f'(\lambda)}{2f'(\mu)f'(x_i)}(x_i-c_i)(x_i-\xi).
\end{align*}
From Definition \ref{Def:Monotonic} and the input of Algorithm
\ref{LZ1}, we know that $f'(\xi) \neq 0$ and $f''(\xi) \neq 0$.
Hence,
\begin{align*}
\lim _{i \rightarrow \infty} \sup
\frac{|x_{i+1}-c_{i+1}|}{|x_{i}-c_{i}|^2} \leq
|\frac{f''(\xi)}{2f'(\xi)}| \neq 0.
\end{align*}
Therefore, the sequence of the lengths of the intervals has at least
quadratic convergence rate.
\end{proof}

\begin{algorithm}[h] \label{LZ2}
  \SetLine
  \KwData{A closed interval $[a,b]$ ($ab>0$ and $a<b$) in a monotonic convex isolation
  of a non-constant real univariate square-free polynomial $f$ and a positive integer
  $L$}

  \KwResult{A closed interval $[a',b']$ such that the only root $\xi$ of $f=0$ in $[a,b]$ is also
  in it and $(b'-a')/|\xi| \leq 10^{-L}$}
  \If{$|b-a| \leq 10^{-L} \min(|a|,|b|)$ }{\KwRet $[a,b]$;}
  \eIf{$f(a)f''(a)>0$ \label{Initial_X_Begin}}{$x \longleftarrow a$; $c \longleftarrow b$; }
  {$x \longleftarrow b$; $c \longleftarrow a$;}\label{Initial_X_End}

  $u \longleftarrow f(x)$; $v \longleftarrow f(c)$\;
  $z \longleftarrow c-v/f'(c)$\;
  \While{$z \not \in [a,b]$ \label{loop2begin}  }
  {$c \longleftarrow (xv-cu)/(v-u)$\;
  $v \longleftarrow f(c)$\;
  $z \longleftarrow c-v/f'(c)$\;} \label{loop2end}

  $x \longleftarrow z$\;

  \While{$|x-c|> 10^{-L} \min(|x|,|c|)$ \label{loop3begin}}{
    $u \longleftarrow f(x)$\;
    $c \longleftarrow (xv-cu)/(v-u)$\;
    \If{$|x-c| \leq 10^{-L} \min(|x|,|c|)$}{\KwRet $[\min(x,c),\max(x,c)]$;}
    $v \longleftarrow f(c)$\;
    $x \longleftarrow c-v/f'(c)$\;
    }\label{loop3end}

  \KwRet $[\min(x,c),\max(x,c)]$.
  \caption{LZ2}
\end{algorithm}

\begin{example} \label{Example:Interval_LZ2}
For the polynomial and the initial interval in Example
\ref{Example:LZ1}, we apply Algorithm \ref{LZ2} to them and get the
intermediate interval sequence (other than Example
\ref{Example:LZ1}, the intervals renew only one endpoint each time):

\begin{align*}
[a_0,b_0] & = [\frac{4389}{1024},\frac{1097}{256}] \\
[a_0,b_1] & = [\frac{4389}{1024},\frac{1261419417}{294336896}] \\
[a_1,b_1] & = [\frac{6671209230324943307293}{1556645655550311117184}, \frac{1261419417}{294336896}]\\
[a_1,b_2] & = [\frac{6671209230324943307293}{1556645655550311117184},\\
          & \frac{283700456965465533230109 \cdots \textup{$42$ digits are omited}}{66198056239039164770905 \cdots \textup{$42$ digits are omited}}]
\end{align*}

$(b_2-a_1)/b_2 \approx 0.1438320660 \times 10^{-10} < 10^{-8}$.
Hence, the algorithm Algorithm \ref{LZ2} terminates and outputs
$[a_1,b_2]$.
\end{example}

\begin{theorem} \label{Theorem:LZ2}
The iteration in Algorithm \ref{LZ2} converges and the convergence
rate is at least 3.
\end{theorem}

\begin{proof}
The iteration in Algorithm \ref{LZ2} has two stages, i.e., the loops
in lines \ref{loop2begin}-\ref{loop2end} and in lines
\ref{loop3begin}-\ref{loop3end}.

We first prove that the loop in lines
\ref{loop2begin}-\ref{loop2end} terminates. In this process, $x$ is
fixed. By Jessen's inequality, we have that $f(c_i)f''(a) <0$ for
the $i$-th iteration of $c$ in this loop (cf. the proof of Theorem
\ref{Theorem:LZ1}). Hence, these $c_i$'s all locate at the other
side of $x$ w.r.t. $\xi$ and if the loop does not terminate the
sequence consisting of them monotonically converges to $\xi$. Thus,
the sequence $\{z_i\}_{i=0}^{\infty}$ will also converge to $\xi \in
(a,b)$, a contradiction. Therefore, the loop in lines
\ref{loop2begin}-\ref{loop2end} terminates and the number of
iterations is independent of $L$.

Next, we study the loop in lines \ref{loop3begin}-\ref{loop3end}
which is the iteration as
\begin{align}
x_{i+1} &= c_i - \frac{f(c_i)}{f'(c_i)} \label{Iteration:Newton_LZ2} \\
c_{i+1} &= \frac{c_i f(x_{i+1}) - x_{i+1} f(c_i)}{f(x_{i+1})-f(c_i)}
\label{Iteration:Secant_LZ2}
\end{align}
where $c_0$ is the last value of $c$ in the loop in lines
\ref{loop2begin}-\ref{loop2end} and $x_0$ is the $x$ assigned in
lines \ref{Initial_X_Begin}-\ref{Initial_X_End}. Consequently, $x_1$
locates in the open interval between $x_0$ and $\xi$. By Jessen's
inequality and from iteration (\ref{Iteration:Secant_LZ2}) we know
that $c_1$ locates in the open interval between $c_0$ and $\xi$.
Note that the derivative of the function $x(t):=t-f(t)/f'(t)$ is
$f(t)f''(t)/(f'(t))^2$ and that $f(c_0)f''(c_0)<0$. By induction on
$i$, it is not difficult to see that the two sequences
$\{x_i\}_{i=0}^{\infty}$ and $\{c_i\}_{i=0}^{\infty}$ all
monotonically converge to $\xi$ but from opposite directions. This
proves that the iteration in Algorithm \ref{LZ2} converges to $\xi$.

At last, we study the convergence rate of this iteration. It is in
fact the convergence rate of the second stage, i.e., the loop in
lines \ref{loop3begin}-\ref{loop3end}. Then we have that
\begin{align*}
x_{i+1}-c_{i+1} &= c_i - \frac{f(c_i)}{f'(c_i)}-\frac{c_i f(x_{i+1})
- x_{i+1}f(c_i)}{f(x_{i+1}) - f(c_i)}\\
                &=
                f(c_i)(\frac{1}{\frac{f(x_{i+1})-f(c_i)}{x_{i+1}-c_i}}-\frac{1}{f'(c_i)}).
\end{align*}
Expand $f(x_{i+1})$ at $c_i$ as
$f(c_i)+f'(c_i)(x_{i+1}-c_i)+f''(\lambda)(x_{i+1}-c_i)^2 /2$ where
$\eta$ is in the open interval between $c_i$ and $x_{i+1}$. Then,
\begin{align*}
x_{i+1}-c_{i+1} &= -f^2(c_i)\frac{f''(\lambda)}{2f'(\eta)f'(c_i)}\\
                &= -\frac{f'(\tau)^2f''(\lambda)}{2f'(\eta)f'(c_i)}(c_i-\xi)^2
\end{align*}
where $\tau \in (\min(c_i,\xi),\max(c_i,\xi))$ and $\eta \in
(\min(x_{i+1},c_i),\\ \max(x_{i+1},c_i))$. Hence, $\{|x_i -
c_i|\}_{i=0}^{\infty}$ and $\{ |c_i - \xi| \}_{i=0}^{\infty}$ have
the same convergence rate around $\xi$. Now, we study the latter.
Again, $f(x_{i+1})$ can be expanded at $c_i$ as
$f(c_i)+f'(c_i)(x_{i+1}-c_i)+f''(c_i)(x_{i+1}-c_i)^2
/2+f'''(\kappa)(x_{i+1}-c_i)^3 /6$ and $f(\xi)$ ($=0$) can be
expanded as $f(c_i) + f'(c_i)(\xi -c_i) + f''(\sigma)(\xi-c_i)^2/2$
and $f(c_i) + f'(c_i)(\xi -c_i) +
f''(c_i)(\xi-c_i)^2/2+f'''(\theta)(\xi-c_i)^3/6$. Then, from
\mbox{iteration (\ref{Iteration:Secant_LZ2})} we have that
\begin{align*}
c_{i+1}-\xi &=
\frac{(c_{i}-\xi)f(x_{i+1})-(x_{i+1}-\xi)f(c_{i})}{f(x_{i+1})-f(c_{i})}\\
            &= \frac{x_{i+1}-c_i}{f(x_{i+1})-f(c_i)}
(-f(c_i)+(c_i-\xi) \times
\\ & (f'(c_i)+\frac{f''(c_i)}{2}(x_{i+1}-c_i)+\frac{f'''(\kappa)}{6}(x_{i+1}-c_i)^2
)) \\
   &=\frac{1}{f'(\eta)}
   (-f(c_i)+(c_i-\xi)f'(c_i)- \\
   & \frac{f''(c_i)}{2f'(c_i)}f(c_i)(c_i-\xi)+
   \frac{f'''(\kappa)f'(\tau)^2}{6f'(c_i)^2}(c_i-\xi)^3)\\
   &=\frac{1}{f'(\eta)} (\frac{f''(c_i)}{2f'(c_i)}(c_i-\xi)((c_i-\xi)f'(c_i) -f(c_i)) + \\
   & \frac{1}{6}(f'''(\theta)+\frac{f'''(\kappa)f'(\tau)^2}{f'(c_i)^2})(c_i-\xi)^3) \\
   &= \frac{1}{f'(\eta)} (\frac{f''(c_i)f''(\sigma)}{4f'(c_i)} + \\
   & \frac{1}{6}(f'''(\theta)+\frac{f'''(\kappa)f'(\tau)^2}{f'(c_i)^2}))(c_i-\xi)^3.
\end{align*}
Therefore, $$\lim _{i \rightarrow \infty} \frac{ |c_{i+1} -
\xi|}{|c_{i} - \xi|^3} =
|\frac{3f''(\xi)^2+4f'(\xi)f'''(\xi)}{12f'(\xi)^2}|$$ which implies
that the sequence $\{c_i-\xi\}_{i=0}^{\infty}$ has at least cubic
convergence rate, and so does the interval iteration using
(\ref{Iteration:Newton_LZ2}) and (\ref{Iteration:Secant_LZ2}).
\end{proof}

\begin{remark}
The efficiency index of algorithms LZ1 and LZ2 are $\sqrt[3]{2}$ and
$\sqrt[3]{3}$, respectively. Hence, in the sense of efficiency index
LZ2 is also more efficient than LZ1.
\end{remark}

\section{Speed-up} \label{Speedup}
In this section, we study how to speed up LZ1 and LZ2 by using the
floating-point interval method.

As we have seen in Example \ref{Example:Interval_LZ2}, the sizes of
the integers representing the intervals increase dramatically. Given
$f \in \mathbb{Q}[x]$ and $p/q \in \mathbb{Q}$, then $f(p/q) \in
\mathbb{Q}$. In general, the maximal length of the numerator and
denominator of $f(p/q)$ approximates to
$\deg(f)\max(\textup{length}(p),\textup{length}(q))$. Hence, if
$\deg(f)$ is large, the computation will become quite time and
memory consuming.

Numerical computation with floating-point numbers can avoid this
difficulty; however, we cannot know exactly whether a numerical
result in common sense is reliable or not, i.e., its correct bits is
unclear or whether it is zero is unclear.

If floating-point numbers have definite representations, it is
possible to estimate the bounds of a numerical result
\cite{Rump88,Rump10,JohnsonKrandick97,CJK02}. Hence, for a nonzero
real number, a narrow interval with floating-point endpoints can
show its correct bits. But zero recognition remains an ill-posed
problem.

Note that we have removed all the ill-posed cases in the former
sections, i.e., the square-free decomposition to get multiplicities
and remove common zeros of $f$ anf $f'$ (cf. Section
\ref{Preliminaries}), the local monotonic convex decomposition to
remove common zeros of $f$ and $f''$ (cf. Section \ref{Monotonic
convex isolation}), the restriction $ab >0$ of input intervals of
LZ1 and LZ2 to remove zero roots of $f=0$ (cf. Section
\ref{Refinement}). As a result, the floating-point interval method
can be applied to evaluate the bounds of $f(p/q)$ (with Horner
scheme) and decide its sign in LZ1 and LZ2.

There exist several packages for the floating-point interval
computation, e.g. Rump's \verb"INTLAB/Matlab" package
\cite{Rump10,INTLAB}, Geulig, Kraemer and Grimmer's
\verb"intpack/Maple" package \cite{GKG} based on the package
\verb"intpak/Maple" \cite{CorlessConnell93} developed by Corless and
Connell, and Revol and Rouillier's MPFI open-source C library
\cite{RS}. We use the functions in the \verb"intpack/Maple" package,
because this package can deal with floating-point numbers with
arbitrary lengths and our algorithms need other functions in
\verb"Maple".

To apply the functions in \verb"intpacX", we should transform LZ1
and LZ2 into their floating-point interval versions and check the
computation step by step. Since many similar details should be
concerned, it is not suitable to present the pseudo codes here. So
we give some principles as follows.
\begin{enumerate}
  \item Pick the initial length $l$ (decimal digits) of floating-point numbers, e.g.
take $l := \max (\min(100,\deg(f) +5), \textup{digits of } a,
\textup{digits of } b)$.
  \item Replace every arithmetic in algorithms LZ1 and LZ2 by
corresponding interval arithmetic in \verb"intpakX".
  \item If an interval value of a polynomial at a point contains zero, then increase $l$ (e.g. $l :=
  2l$ for LZ1) and repeat relative computation until the endpoints of the new interval value share a fixed sign.
  \item Rewrite the inequality conditions in the loops and the first lines of the algorithms into interval versions,
   so that the algorithms can obtain stronger results than in the exact case.
\end{enumerate}

A natural question is whether the super-linear convergence rates can
be retained or not when we use the interval method with
floating-point numbers instead of exact computation with rational
numbers. The answer is ``YES". The proof of this claim should deal
with many details including why LZ1 and LZ2 can avoid ill-posed
cases. We omit most of them and only show the key insight. Suppose
that $[u,v]$ is an interval containing $\xi$ and $u$, $v$ are
included in two narrow enough intervals $[u_*,u^*]$ and $[v_*,v^*]$,
respectively, where $u_*$, $u^*$, $v_*$ and $v^*$ are all
floating-point numbers. Then, $\xi \in [u^*,v_*] \subset [u,v]$. To
determine whether the two intervals are enough narrow, we should
check the signs of the values of the polynomial at the endpoints of
$[u_*,u^*]$ and $[v_*,v^*]$. However, $[u_*,v^*]$ is a better choice
than $[u^*,v_*]$ in practice, since the former is less possible to
obtain a false isolating interval and hence makes the algorithms
more efficient than the latter.

\begin{example} \label{Example:Numerical_LZ2}
Consider the numerical version of Example
\ref{Example:Interval_LZ2}. Take the initial length of floats $l :=
\max (3+5,[1.6 \times 8])=12$. Then, the numerical LZ2 yields
\begin{align*}
[a_0,b_0] &=[ \underline{4.28}515625000,\underline{4.28}613281250], \\
[a_0,b_1] &=[ \underline{4.285}15625000,\underline{4.285}63130935], \\
[a_1,b_1] &=[\underline{4.285631}226694662183,\underline{4.285631}30935],\\
[a_1,b_2] &=[ \underline{4.285631226}694662183, \\
          &  \underline{4.285631226}70901127793655266277], \textup{ and} \\
(b_2 -a_1)/a_1 &= 0.334818704119335811719986541959 \times 10^{-11} \\
          &< 10^{-8}.
\end{align*}
To see the convergence rate, we continue to compute
\begin{align*}
[a_2,b_2] &=[\underline{4.285631226709011277936}4772441617276999\\
          & 1330501894, \underline{4.285631226709011277936}55266277].
\end{align*}
The correct digits of $[a_1,b_1]$ and $[a_2,b_2]$ are $7$ and $22$,
respectively. This shows that the iteration has cubic convergence
rate.
\end{example}
We can find that the expression of $b_2$ in Example
\ref{Example:Numerical_LZ2} is much shorter than that in Example
\ref{Example:Interval_LZ2} and that the swell of coefficients of
polynomials has been solved.

\section{Experiments} \label{Experiments}
To show the effectiveness of the algorithms LZ1 and LZ2, we
implemented them in \verb"Maple15" with the open source package
\verb"intpakX" \cite{GKG} and compared our implementations with the
\verb"Maple" function \verb"RefineBox" in the package
\verb"RegularChains" \cite{BCLM09}. All the experiments in this
section were done on a computer with Intel(R) Core(TM) i3-2100 CPU @
3.10GHz.

Chebyshev polynomials of the first kind were used as the tested
polynomials and were generated by the \verb"Maple" function
\verb"ChebyshevT". The testing results are listed in Tables
\ref{table1}-\ref{table4}. In each table, $L$ is the number of
correct digits of the output (cf. Algorithms \ref{LZ1} and Algorithm
\ref{LZ2}), ``ratio1" and ``ratio2" are the rounded time ratios
RefineBox/LZ1 and RefineBox/LZ2, respectively. For each polynomial,
we aimed to refining the isolating interval to a precision
$10^{-L}$. The CPU times in the tables were tested by the
\verb"Maple" function \verb"time".

For a polynomial $f$ with degree $n$ in Table \ref{table1}, Table
\ref{table3} and Table \ref{table4}, we isolated all its real zeros
by the function \verb"RealRootIsolate"\footnote{\scriptsize There
are three other methods for isolating real roots in Maple15. They
are \emph{realroot}, \emph{RealRootIsolate} with the option\emph{
method='Discoverer'} \cite{XiaYang02,XiaZhang06,DISCOVERER}, and
\emph{Isolate} \cite{RS,Rouillier99,RZ04} in the \emph{RootFinding}
package.}\cite{BCLM09} with the option \verb"'rerr'=1/2", picked the
$(n/2)$-th ``box", and refined the isolating interval until its
width was no larger than $10^{-5}$ by
\verb"RefineBox"\footnote{\scriptsize According to \mbox{Algorithm
6} in \cite{BCLM09}, this function performs a generalization of the
bisection method; but in this step, it took less than $1$ second in
all tested cases.} and viewed the result as the initial isolating
interval. For Table \ref{table3} and Table \ref{table4}, this result
is $[242345/262144, 484695/524288]$. Each $(n/2)$-th ``box" happened
to contain a zero of the largest irreducible factor $g$ in degree of
$f$ and contain no zeros of $g'$ or $g''$. We denote the degree of
this factor by $n^*$.

Factorizations for the polynomials in Table \ref{table2} were
time-consuming. By Algorithm \ref{LMCD}, we knew that these
polynomials themselves compose their local monotonic convex
decompositions. Consequently, they all have monotonic convex
isolations according to Definition \ref{Def:Local}. For a polynomial
$f$ in Table \ref{table2}, we used \verb"Isolate" (developed by
Fabrice Rouillier in C language) in the \verb"RootFinding" package
instead of \verb"RealRootIsolate", because the former is faster than
the latter when $n$ is large and \verb"RefineBox" need not be called
for the experiments in this table. The output intervals of
\verb"Isolate" with the option \verb"output=interval" were narrow
enough that $f'$ and $f''$ have no zeros and can be used as the
input intervals of LZ1 and LZ2. We picked the $(n/2)$-th element in
the output of \verb"Isolate".

From all of these tables, we can see that our implementations of LZ1
and LZ2 are much faster than the \verb"Maple" function
\verb"RefineBox" and in most cases LZ2 was more efficient than LZ1.
Since the environment variable \verb"Digits" in our implementations
of LZ1 and LZ2 did not change continuously, the time costs would
have jumps in the tables. Hence, it is not confusing that sometimes
LZ1 behaved a little better than LZ2 in our experiments.
\begin{table}[!h]
\caption{Timings (s) ($L=1000$) \label{table1}} \centering
\begin{tabular}{|c|c|c|c|c|c|c|}
\hline $n$& $n^*$ & RefineBox & LZ1   & ratio1& LZ2   & ratio2 \\
\hline 100& 80& 89.560   & 1.965 & 46    & 1.794 & 50   \\
\hline 200& 160& 347.695  & 4.492 & 77    & 3.900 & 89   \\
\hline 300& 160& 349.317  & 4.383 & 80    & 2.262 & 154   \\
\hline 400& 320& 1414.819 & 5.428& 261    & 5.974 & 237   \\
\hline 500& 400& 2232.982 & 11.076& 202    & 6.583 & 309   \\
\hline 600& 320& 1427.814 & 19.890& 72    & 5.007 & 285   \\
\hline 700& 480& 3370.198 & 16.770& 201    & 8.018 & 420   \\
\hline 800& 640& 5898.132 & 21.387& 276    & 10.920 & 540   \\
\hline 900& 480& 3332.259 & 14.757& 226    & 7.987 & 417   \\
\hline
\end{tabular}

\caption{ More comparisons ($L= 1000$) \label{table2}} \centering
\begin{tabular}{|c|c|c|c|c|c|c|c|}
\hline $n$ & 1000 & 1200   &  1400  & 1600   &  1800   &2000\\
\hline LZ1& 31.075& 27.284 & 44.928 & 27.346 &  59.061 & 34.398\\
\hline LZ2& 10.888& 18.626 & 15.880 & 25.958 &  29.764 & 33.275\\
\hline
\end{tabular}

\caption{Timings (s) ($n= 1000,n^*=800$) \label{table3}} \centering
\begin{tabular}{|c|c|c|c|c|c|}
\hline $L$ &  RefineBox & LZ1 & ratio1 & LZ2 & ratio2 \\
\hline 100&  55.629    &5.506 &  10    &3.978 & 14   \\
\hline 200&  217.075   &9.874&  22   &5.428 & 40   \\
\hline 300&  540.262   &11.887&  45   &5.974 & 90   \\
\hline 400& 1067.062   &11.934&  89   &5.787 & 184   \\
\hline 500& 1809.892   &14.040&  129   &8.361& 216   \\
\hline 600& 2792.495   &13.915&  201  &12.480& 224   \\
\hline 700& 4103.359   &13.884&  296  &12.558& 327  \\
\hline
\end{tabular}

\caption{ More comparisons ($n= 1000,n^*=800$) \label{table4}}
\centering
\begin{tabular}{|c|c|c|c|c|c|c|}
\hline $L$&800    & 900    & 1000   &  2000 & 3000   \\
\hline LZ1&14.008 & 19.578 & 48.297 &48.578& 63.679  \\
\hline LZ2&12.448 & 12.542 & 12.542 &27.705& 27.612 \\
\hline
\end{tabular}
\end{table}
\section{Conclusion}
In this paper, we provided a quadratically convergent algorithm LZ1
and a cubically convergent algorithm LZ2 to refine real roots of
univariate polynomial equations. Before applying LZ1 or LZ2, the
polynomial should be make the square-free decomposition and the
local monotonic convex decomposition (LMCD), so that every component
polynomial has a monotonic convex isolation (MCI). Moreover, we used
the interval method with floating-point numbers to estimate the
values of polynomials at rational points to improve the efficiency
the algorithms. Experiments on benchmark polynomials showed that if
we need high precisions of the real roots, then both of the two
algorithms are much faster than the function \verb"RefineBox" in
\verb"Maple15".

\section{Acknowledgments}
The author would like to thank Professor Lihong Zhi, Professor Na
Lei and Doctor Wei Niu for helpful discussions.

This work has been supported by a NKBRPC 2011CB302400, the Chinese
National Natural Science Foundation under Grants 91118001,
60821002/F02, 60911130369 and 10871194.

\bibliographystyle{abbrv}
\bibliography{rrr}

\end{document}